\documentclass[12pt]{article}

\usepackage[all,cmtip]{xy}

\usepackage{amsmath}
\usepackage{amsthm}
\usepackage{amsfonts}
\usepackage{amssymb}
\usepackage{graphicx}
\usepackage{subfigure}
\usepackage{multirow}
\usepackage{color}
\usepackage{blkarray}
\usepackage{natbib}
\usepackage[dvipsnames]{xcolor}

\newcommand{\bm}[1]{\boldsymbol{#1}}

\renewcommand{\d}{\mathrm{d}}
\renewcommand{\P}{\mathbb{P}}

\newcommand{\Y}{\bm{Y}}
\newcommand{\E}{\mathbb{E}}

\newcommand{\mB}{\mathcal{B}}
\newcommand{\mR}{\mathcal{R}}

\newcommand{\mK}{\mathcal{K}}

\newcommand{\mX}{\mathcal{X}}
\newcommand{\mT}{\mathcal{T}}
\newcommand{\wt}[1]{\widetilde{#1}}
\def\mS{\mathcal{S}}

\newcommand{\Var}{\mathsf{Var\,}}
\newcommand{\iid}{\stackrel{iid}{\sim}}

\newcommand{\wh}[1]{\smash{\widehat{#1}}}

\def\C {\,|\:}

\def\C {\,|\:}

\def\a{\bm{a}}
\def\mF{\mathcal{F}}

\def\b{\bm{\beta}}
\def\Y{\bm{Y}}

\def\x{\bm{x}}
\def\y{\bm{y}}

\def\mV{\mathcal{V}}

\def\mE{\mathcal{E}}

\def\BO{\mT}

\def\b{\bm{\beta}}

\renewcommand{\d}{\mathrm{d}}
\newcommand{\e}{\mathrm{e}}

\renewcommand{\1}{\mathbb{I}}
\newcommand{\N}{\mathbb{N}}
\newcommand{\R}{\mathbb{R}}
\newcommand{\Ha}{\mathcal{H}^\alpha}
\newcommand{\Haa}{\mathcal{H}^\gamma}

\usepackage{tikz}
\usepackage{tikz-qtree}
\usetikzlibrary{positioning,shadows,arrows}

\newtheorem{definition}{Definition}[section]
\newtheorem{lemma}{Lemma}[section]

\newtheorem{theorem}{Theorem}[section]

\newtheorem{remark}{Remark}[section]

\def\spacingset#1{\renewcommand{\baselinestretch}%
{#1}\small\normalsize}
\numberwithin{equation}{section}

\spacingset{1}

\title{\sf On Semi-parametric Bernstein-von Mises \\
Theorems for BART}

\author{Veronika Ro\v{c}kov\'{a}\footnote{\sl Assistant Professor in Econometrics and Statistics and  James S. Kemper Faculty Scholar  at the University of Chicago Booth School of Business. This research was supported by the James S. Kemper Faculty Research Fund.}}

\begin{document}

\maketitle

\begin{abstract}
Few methods in Bayesian non-parametric statistics/ machine learning have received as much attention as
Bayesian Additive Regression Trees (BART). 
While BART is now routinely performed for  prediction tasks,  its theoretical properties   began to be understood only very recently.
In this work, we continue the theoretical investigation of BART initiated by \cite{rockova_vdp}.
In particular, we study the Bernstein-von Mises (BvM) phenomenon (i.e. asymptotic normality) for smooth linear functionals of the  regression surface  within the framework of non-parametric regression with fixed covariates. 
As with other adaptive priors,  the BvM phenomenon may fail when the regularities of the functional and the truth are not compatible.
To overcome the curse of adaptivity under hierarchical priors, we induce a self-similarity assumption to ensure convergence towards a single Gaussian distribution as opposed 
to a Gaussian mixture.
Similar qualitative restrictions on the functional parameter are known to be necessary for adaptive inference. 
Many machine learning methods lack coherent probabilistic mechanisms for gauging uncertainty.
BART readily provides such quantification via posterior credible sets. The BvM theorem implies that the credible sets are also confidence regions with the same asymptotic coverage.  
This paper presents the first asymptotic normality result for BART priors, providing another piece of evidence that BART is a valid tool from a frequentist point of view.
\end{abstract}

\clearpage

\spacingset{1.3}

\section{Introduction}
This note studies the Gaussian approximability of certain aspects of posterior distributions  in  non-parametric regression with trees/forest priors.
 Results of this type, initially due to \cite{laplace} but most commonly known as Bernstein-von Mises (BvM) theorems, imply that posterior-based inference asymptotically coincides 
 with the one based on traditional frequentist  $1/\sqrt{n}$-consistent estimators. In this vein, BvM theorems provide a rigorous frequentist justification of Bayesian inference. 
The main thrust of this work is to understand the extent to which  this phenomenon holds for various incarnations of Bayesian Additive Regression Trees (BART) of  \cite{bart},  one of the workhorses of Bayesian machine learning.

In simple words, the BvM phenomenon  occurs when, as the number of observations increases, the posterior distribution has approximately the shape of a Gaussian distribution centered at an efficient estimator of the parameter of interest. 
 Moreover, the posterior credible sets,  i.e. regions with prescribed posterior probability, are then {\sl also} confidence regions with the same asymptotic coverage.
This property has nontrivial practical implications in that constructing confidence regions from MCMC simulation can be relatively straightforward whereas computing confidence regions directly is not. Under fairly mild assumptions, BvM statements  can be expected to hold in regular finite-dimensional models \citep{vdv_book}.

  
Unfortunately, the frequentist theory on asymptotic normality does not generalize fully to semi-  and non-parametric estimation problems \citep{bickel_kleijn}. 
 Freedman initiated the discussion on the  consequences of unwisely chosen priors in the 1960's, providing a negative BvM result in a basic $\ell_2$-sequence Gaussian conjugate setting.
The warnings against seemingly  innocuous priors that may lead to  posterior inconsistency were then reiterated  many times in the literature, including  \cite{cox_93} and \citep{diaconis_freedman1,
diaconis_freedman2}.  Other counterexamples and anomalies of the BvM behavior in infinite-dimensional problems can be  found in  \cite{Johnstone10} and \cite{leahu}.
While, as pointed out by \cite{freedman_1999},  analogues of the BvM property for infinite-dimensional parameters are not immediately obvious,  rigorous notions  of non-parametric BvM's
 have been  introduced in several pioneering works \citep{leahu,ghosal_2000,castillo_nickl1,castillo_nickl2}.

Unwisely chosen priors leave room for unintended consequences also in semi-parametric contexts  \citep{riv_rous}. \cite{castillo_bias} provided an interesting counterexample where the posterior does not display the BvM behavior due to a non-vanishing bias in the centering of the posterior distribution.
Various researchers have nevertheless documented  instances of the BvM limit  in  semi-parametric models (a) when the parameter can be separated into a finite-dimensional parameter of interest and an infinite-dimensional nuisance parameter \citep{castillo_gp, shen_2002,bickel_kleijn,cheng_kosorok, Johnstone10, jonge_zanten}, and (b) when the parameter of interest is a functional of the infinite-dimensional parameter \citep{riv_rous,castillo_rous}. In this work, we focus on the latter class of semi-parametric BvM's and study the asymptotic behavior of smooth linear functionals of the regression function.

We consider the standard non-parametric regression setup, where a vector  of responses $\Y^{(n)}=(Y_1,\dots,Y_n)'$ is linked to fixed (rescaled) predictors $\x_i=(x_{i1},\dots,x_{ip})'\in[0,1]^p$ for each $1\leq i\leq n$ through
\begin{equation}\label{model}
Y_i=f_0(\x_i)+\varepsilon_i,\quad \varepsilon_i\iid\mathcal{N}(0,1),
\end{equation}
where $f_0$ is an unknown $\alpha$-H\"{o}lder continuous function on a unit cube $[0,1]^p$. The true generative model  giving rise to \eqref{model} will be denoted with $\P_0^n$.

Each model is parametrized by $f\in\mF$, where $\mF$ is an infinite-dimensional set of possibilities of $f_0$. Let $\Psi:\mF\rightarrow \R$ be a measurable functional of interest and let $\Pi$ be a probability distribution on $\mF$. Given observations $\Y^{(n)}$ from  \eqref{model}, we study the asymptotic behavior of the posterior distribution of $\Psi(f)$, denoted with $\Pi[\Psi(f)\C\Y^{(n)}]$. Let $\mathcal{N}(0,V)$ denote the centered normal law with a covariance matrix $V$. In simple words, we want to show that under the Bayesian CART or BART priors on $\mF$, the {\sl posterior distribution satisfies the BvM-type property} in the sense that
\begin{equation}\label{BvM}
\Pi[\sqrt{n}\left(\Psi(f)-\wh\Psi\right)\C\Y^{(n)}]\rightsquigarrow N(0,V)
\end{equation}
as $n\rightarrow\infty$ in $\P_0^n$-probability, where $\wh\Psi$ is a random centering point estimator  and where $\rightsquigarrow$ stands for weak convergence. We make this statement more precise in Section \ref{sec:rudi}

\cite{castillo_rous} provide general conditions on the model and on the functional $\Psi$ to guarantee that \eqref{BvM} holds. These conditions describe how the choice of the prior influences the bias $\wh\Psi$ and variance $V$.  Building on this contribution, we provide (a) tailored statements for various incarnations of tree/forest priors that have  occurred in the literature, (b) extend the considerations to {\sl adaptive} priors under self-similarity.

\subsection{Notation}
The notation $\lesssim$ will be used to denote  inequality up to a constant, {$a\asymp b$ denotes $a\lesssim b$ and $b\lesssim a$ and $a\vee b$ denotes $\max\{a,b\}$}. The $\varepsilon$-covering number of a set $\Omega$ for a semimetric $d$, denoted by $N(\varepsilon,\Omega,d),$ is the minimal number of $d$-balls of radius $\varepsilon$ needed to cover set $\Omega$. We denote by $\phi(\cdot;\sigma^2)$ the normal density with  zero mean and  variance $\sigma^2$ 
With $||\cdot||_2$ we denote the standard Euclidean norm and with
%
$\lambda_{min}(\Sigma)$ and $\lambda_{max}(\Sigma)$ the extremal eigenvalues of a matrix $\Sigma$. The set of $\alpha$-H\"{o}lder continuous functions  (i.e. H\"{o}lder smooth with $0<\alpha\leq1$) on $[0,1]^p$ will be denoted with $\Ha_p$.

\section{Rudiments}\label{sec:rudi}
Before delving into our development, we first we make precise the definition of asymptotic normality.
\begin{definition}
We say that the posterior distribution for the functional $\Psi(f)$ is asymptotically normal with centering $\Psi_n$ and variance $V$ if, for $\beta$ the bounded Lipschitz metric for weak convergence, and  the real-valued mapping $\tau_n:\sqrt{n} [\Psi(f)-\Psi_n]$, it holds that 
\begin{equation}\label{beta:convergence}
\beta(\Pi[\cdot\C\Y^{(n)}]\circ \tau_n^{-1},\mathcal{N}(0,V))\rightarrow 0
\end{equation}
in $\P_0^n$ probability as $n\rightarrow\infty$, which we denote as $\Pi[\cdot\C\Y^{(n)}]\circ \tau_n^{-1}\rightsquigarrow\mathcal{N}(0,V)$.
\end{definition}

When efficiency theory at the rate $\sqrt{n}$ is available,  we say that the posterior distribution satisfies the BvM theorem if  \eqref{beta:convergence} holds with $\Psi_n=\wh\Psi+o_P(1/\sqrt{n})$ for $\wh\Psi$ a linear efficient estimator of $\Psi(f_0).$
The proof of such a statement typically requires of a few steps (a) a semi-local step  where one proves that the posterior distribution concentrates on certain sets and (b) a strictly local study on these sets, which can be carried out under the LAN (local asymptotic normality) conditions.
Denoting the log-likelihood  with
$$
\ell_n(f)=\frac{n}{2}\log 2\pi-\sum_{i=1}^n\frac{[Y_i-f(\x_i)]^2}{2},
$$
we define  the log-likelihood ratio  $\Delta_{n}(f)=\ell_n(f)-\ell_n(f_0)$  and express it  as a sum of a quadratic term and  a stochastic term via the LAN  expansion
as follows
\begin{align}
\Delta_n(f)
&=-\frac{n}{2}\|f-f_0\|_L^2+\sqrt{n}W_n(f-f_0),
\end{align}
where 
\begin{align*}
\|f-f_0\|_L^2&=\frac{1}{n}\sum_{i=1}^n[f_0(\x_i)-f(\x_i)]^2,\\
W_n(f-f_0)&=\langle f-f_0,\sqrt{n}\,\bm{\varepsilon}\rangle_L=\frac{1}{n}\sum_{i=1}^n\sqrt{n}\,\varepsilon_i\,[f(\x_i)-f_0(\x_i)].
\end{align*}
Recall that the phrase ``semi-parametric" here refers  to the  problem of estimating functionals in an infinite-dimensional model rather than Euclidean parameters in the presence of  infinite-dimensional nuisance parameters. In this note, we consider the smooth linear functional 
\begin{equation}\label{psi_functional}
\Psi(f)=\frac{1}{n}\sum_{i=1}^n a(\x_i)f(\x_i)
\end{equation}
for some  smooth uniformly bounded $\gamma$-H\"{o}lder continuous function $a(\cdot)$, i.e. $\|a\|_\infty< C_a$ and $a\in\mathcal{H}^\gamma_p$ for some $0<\gamma\leq1$.  Were $\gamma>1$, H\"{o}lder continuity would imply that $a(\cdot)$ is a constant function and \eqref{psi_functional} boils down to a constant multiple of the average regression surface evaluated at fixed design points. This quantity is actually of independent interest and has been studied in a different setup by \cite{ray_vdv} who focus on the posterior distribution of the ``half average treatment effect estimator" (the mean regression surface) in the presence of missing data and random covariates.
Besides its standalone purpose, the semi-parametric BvM result for functionals \eqref{psi_functional} has a useful alternative purpose in the multiscale approach to posteriors towards obtaining non-parametric BvM statements  for BART priors \citep{castillo_rockova}. One of the ingredients is proving convergence of  finite-dimensional distributions, which can be shown by the Cramer-Wald theorem  applying results presented here.

\section{Tree and Forest Priors}
Regression trees provide a piecewise constant reconstruction of the regression surface  $f_0$, where the pieces correspond to terminal nodes of recursive partitioning \citep{donoho}.
Before introducing the tree function classes, we need to define the notion of tree-shaped partitions.

Starting from a parent node $[0,1]^p$, a binary tree partition is obtained by successively applying a  splitting rule on a chosen internal node.
Each such internal node is divided into two  daughter cells with a split  along one of the $p$ coordinates at a chosen observed value. These daughter cells define two new rectangular subregions of $[0,1]^p$, which can be split further (to become internal nodes) or end up being  terminal nodes. The terminal cells after $K-1$ splits then yield a tree-shaped  partition $\mT=\{\Omega_k\}_{k=1}^K$ consisting of boxes $\Omega_k\subset[0,1]^p$. We denote with $\mV^K$ the set of all tree-shaped partitions that can be obtained by recursively splitting $K-1$ times at observed values $\mX\equiv\{\x_i\}_{i=1}^p$.

The tree functions can be then written as
\begin{equation}\label{tree}
f_{\mT,\b}(\x)=\sum_{k=1}^K\mathbb{I}(\x\in\Omega_k)\beta_k,
\end{equation}
where $\mT=\{\Omega_k\}_{k=1}^K\in\mV^K$ and where $\b=(\beta_1,\dots,\beta_K)'\in\R^{K}$ is a vector of jump sizes. Solitary trees are not as good performers as ensembles of trees/forests \citep{breiman2,bart}.
The forest mapping underpinning the BART method of \cite{bart} is the following sum-of-trees model indexed by a collection of $T$ tree-shaped partitions $\mE=[\mT^1,\dots,\mT^T]$   and step sizes $\mB=[\b^1,\dots,\b^T]$: 
\begin{equation}\label{forest}
f_{\mE,\mB}(\x)=\sum_{t=1}^Tf_{\mT^t,\b^t}(\x)\quad\text{for}\quad \mT^t\in\mV^{K^t}\quad\text{and}\quad \b^t\in\R^{K^t}.
\end{equation}
The prior distribution is assigned  over the class of forests
$$
\mF=\{ f_{\mE,\mB}(\x) \quad\text{of the form \eqref{forest} for some $\mE$ and $\mB$ and $T\in\N$} \}
$$ 
in a hierarchical manner. One first assigns a prior distribution over $T$ (or sets it equal  to some  value) and 
then  a prior over the tree-shaped partitions $\mT^t$  as well as  step sizes $\b^t$ for $1\leq t\leq T$.

\subsection{Tree Partition Priors $\pi(\mT)$}
In 1998,  there were two Bayesian CART developments that surfaced independently of each other: \cite{cart1} and \cite{cart2}.
Albeit related, the two methods differ in terms of the proposed tree partition prior $\pi(\mT)$.

The prior  of \cite{cart2} is equalitarian in the sense that trees with the same number of leaves are a-priori equally likely, regardless of their shape.
To prioritize smaller trees (that do not overfit), one assigns a complexity prior $\pi(K)$ on the tree size (i.e. the number of bottom nodes) $K$. They 
considered the Poisson distribution 
\begin{equation}\label{prior:K}
\pi(K)\,=\, \frac{\lambda^{K}}{(\e^\lambda-1)K!},\,K=1,2,\dots,\quad\text{for some}\quad \lambda>0.
\end{equation}
Given the tree size $K$, one assigns  a uniform prior over  tree topologies 
\begin{equation}\label{prior:partition}
\pi(\mT\C  K)=\frac{\1\left(\mT\in\mV^K\right)}{|\mV^K|},
\end{equation}
where   $|\mV^K|$ is the cardinality of $\mV^K$.
This prior can be straightforwardly  implemented using Metropolis-Hastings strategies  \citep{cart2} and was studied theoretically by \cite{rockova_vdp}.

The Bayesian CART  prior of \cite{cart1}, on the other hand, 
 specifies the prior implicitly as a tree-generating Galton-Watson (GW)  process. This process provides a mathematical representation of an evolving  population of individuals who  reproduce and die subject to laws of chance.
The tree growing process is characterized as follows. Starting with a root node $\Omega_{00}=[0,1]^p$, one decides to split each node $\Omega_{lk}$ into two children by flipping a coin. We are tacitly using the two-index numbering of nodes $(l,k)$,  where $l$ corresponds to the tree layer and $k$ denotes the $(k+1)^{st}$ left-most node at the $l^{th}$ layer. In order to prevent the trees from growing indefinitely, the success probability decays with $l$, making  bushy trees far less likely. Let us denote with $\gamma_{lk}\in\{0,1\}$  the  binary splitting indicator for whether or not a node $\Omega_{lk}$ was divided. \cite{cart1} assume
\begin{equation}\label{eq:split_prob}
\P(\gamma_{lk}=1)=p_{lk}=\frac{\alpha}{(1+l)^\delta}
\end{equation}
for some $\alpha\in(0,1)$ and $\delta>0$.  \cite{rockova_saha} propose a modification $p_{lk}=\alpha^{l}$ for some $1/n<\alpha<1$, which yields optimal posterior concentration in the $\ell_2$ sense. If the node $\Omega_{lk}$ splits (i.e. $\gamma_{lk}=1$), one chooses a splitting rule  by first picking a variable $j$ uniformly from available directions $\{1,\dots, p\}$ and
then picking a split point $c$ uniformly from available data values $x_{1j},\dots, x_{nj}$. 


Unlike in the homogeneous case (where all $\gamma_{lk}$'s are iid), \eqref{eq:split_prob} defines a {\sl heterogeneous} GW process where the offspring distribution is allowed to vary from generation to generation, i.e. the variables $\gamma_{lk}$ are independent but {\sl non-identical}.  Figure \ref{treeproba} depicts the split probabilities.

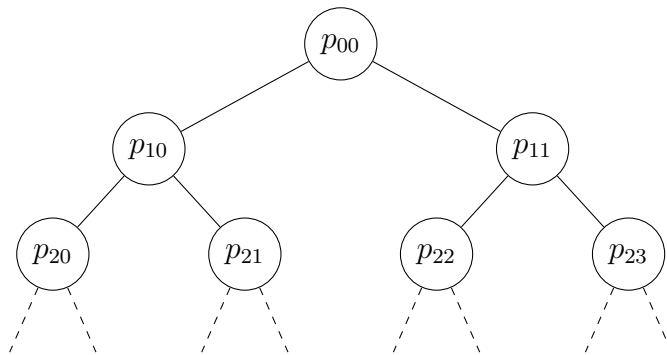
\begin{figure*}[h]
\begin{center}
\begin{tikzpicture}[
   level distance=1.4cm,sibling distance=1cm, 
   edge from parent path={(\tikzparentnode) -- (\tikzchildnode)}]
\Tree [.\node[draw,circle] {$p_{00}$}; 
    \edge node[auto=right] {}; 
    [.\node[draw,circle]{$p_{10}$};  
      \edge node[auto=right] {};  
      [.\node[draw,circle]{$p_{20}$}; 
      \edge[dashed]; {}
      \edge[dashed]; {}
      ]  
      \edge node[auto=left] {}; [.\node[draw,circle]{$p_{21}$}; 
      \edge[dashed]; {} 
      \edge[dashed]; {}
      ] 
          ]
     \edge node[auto=left] {};      
    [.\node[draw,circle]{$p_{11}$};
    \edge node[auto=right] {};  
      [.\node[draw,circle]{$p_{22}$}; 
       \edge[dashed]; {}
      \edge[dashed]; {}
      ] \edge node[auto=left] {}; 
    [.\node[draw,circle]{$p_{23}$};
     \edge[dashed]; {}
      \edge[dashed]; {}
     ] 
    ] ]
\end{tikzpicture}
\end{center}
\caption{Binary tree of prior cut probabilities in Bayesian CART by Chipman et al. (1998)}\label{treeproba}
\end{figure*}

\subsection{Priors on Jump Sizes $\pi(\b\C K)$}
After  partitioning the predictor space into $K$ nested rectangular cells, one needs to assign a prior on the presumed level of the outcome. Throughout this work, we denote with $\b=(\beta_1,\dots,\beta_K)'$ the vector of jump sizes associated with $K$ partitioning cells.
Both \cite{cart1} and \cite{cart2} assumed an independent product of  Gaussians 
$\pi(\b)\sim\prod_{k=1}^K\phi(\beta_k;\sigma^2)$ for some $\sigma^2>0$. \cite{cart3} argue, however, that the simple independence prior on the bottom leaf  coefficients $\beta_k$ may not provide enough structure. They claim that values $\beta_k$ that correspond to nearby cells in the predictor space should be more similar so that the prior incorporates local smoothness. They suggest a prior on bottom leaves  that aggregates priors on the ancestral internal nodes and, in this way, induces correlation among neighboring cells.
Motivated by these considerations, here we allow for  general correlation structures by assuming a multivariate Gaussian prior
\begin{equation}\label{eq:priorb}
\pi(\b\C K)\sim\mathcal{N}_K(\bm 0,\Sigma),\quad\text{with}\quad \lambda_{min}(\Sigma)>c>0\quad\text{and}\quad \lambda_{max}(\Sigma)\lesssim n
\end{equation}
where $\Sigma$ is some $K\times K$ positive definite matrix. We also consider an independent product of Laplace priors (which was not yet studied in this context)
\begin{equation}\label{eq:priorb_laplace}
\pi(\b\C K)=\prod_{k=1}^K \psi(\beta_k;\lambda),
\end{equation}
where $\psi(\beta;\lambda)=\lambda/2\e^{-\lambda|\beta|}$ for some $\lambda>0$.


 \section{Simple One-dimensional Scenario}\label{sec:semi}
To set the stage for our development, we start with a simple toy scenario where (a) $p=1$, (b)  $K$ is regarded as  fixed,  and (c) when there  is {\sl only one} partition $\mT=\{\Omega_k\}_{k=1}^K$ consisting of $K$ {\sl equivalent blocks}.
The  equivalent blocks partition $\mT=\{\Omega_k\}_{k=1}^K$ \citep{anderson1966} comprises   $K$ intervals $\Omega_k$ with roughly equal number of observations in them, i.e. $\mu(\Omega_k)\equiv\frac{1}{n}\sum_{i=1}^n\mathbb{I}(\x_i\in\Omega_k)\asymp 1/K$. For the sake of simplicity, we will also  assume that the data points lie on a regular grid, where $x_i=i/n$ for $1\leq i\leq n$ in which case the intervals $\Omega_k$ have also roughly equal lengths.  This setup was studied previously by \cite{vdp_rockova} in the study of regression histograms.  We relax this assumption in the next section. We denote the class of approximating functions as
 \begin{equation}\label{prior1}
 \mF[K]=\left\{f^K_{\b}:[0,1]^p\rightarrow\R; f^K_{\b}(\x)=\sum_{k=1}^K\1(\x\in\Omega_k)\beta_k\right\},
 \end{equation}
 where we have omitted the subscript $\mT$  and  highlighted the dependence on $K$.  We denote with $\Pi^K(f)$ the prior distribution over $\mF[K]$, obtained by assigning  the prior \eqref{eq:priorb} or \eqref{eq:priorb_laplace}. 
 To further simplify the notation, we will drop the subscript $\b$ and write $f^K$ when referring to the elements of $\mF[K]$.  
 
 The aim is to study the posterior distribution of $\Psi(f^K)$ and to derive conditions under which
 \begin{equation}\label{statement}
 \Pi[ \sqrt{n}(\Psi(f^K)-{\Psi}_n)\C\Y^{(n)}]
 \end{equation}
converges to a normal distribution (in $\P_0^n$ probability) with mean zero and variance $V_0=\|a\|_L^2$, where
 $\Psi_n$ is a random centering, distributed according to a Gaussian distribution with mean $\Psi(f_0)$ and variance $V_0$.

 Using  the fact that convergence of Laplace transforms for all $t$ in probability implies convergence in distribution in probability (Section 1  of \cite{castillo_rous})
this BvM statement holds when $\forall t\in\R$
\begin{equation}\label{laplace}
 \E^{\Pi}[\e^{t\sqrt{n}(\Psi(f^K)-\wh\Psi_K)}\C\Y^{(n)}]\rightarrow \e^{\frac{t^2}{2}\|a\|_L^2}\quad\text{in $\P_0^n$ probability as $n\rightarrow\infty$},
 \end{equation}
where $\wh\Psi_K$ is a linear efficient estimator of $\Psi(f_0)$ such that
 $$ 
 \sqrt{n}(\wh\Psi_K -\Psi_n)=o_P(1).
 $$
In order to show  \eqref{laplace}, we first need to introduce some notation.
Let $a^K$  be the projection of $a$ onto $\mF[K]$, i.e.
$$
a^K(\x)=\sum_{k=1}^K\1(\x\in\Omega_k)a^K_k\quad\text{with}\quad a_{k}^K=\sum_{i=1}^n\frac{\1(\x_i\in\Omega_k)}{n}\frac{a(\x_i)}{\mu(\Omega_k)},
$$
where   $\mu(\Omega_k)$ was defined above as $\mu(\Omega_k)=\frac{1}{n}\sum_{i=1}^n\1(\x_i\in\Omega_k)$. Similarly, we denote with $f_0^K=\sum_{k=1}^K\1(\x\in\Omega_k)\beta^K_k$ the projection of $f_0$ onto $\mF[K]$
with jump sizes $\b^K=(\beta_1^K,\dots,\beta_K^K)'$.
Next, we define \begin{align*}
\wh\Psi_K&=\Psi\left(f_{0}^{K}\right)+\frac{W_n(a^K)}{\sqrt n}\quad \text{and}\quad \Psi_n=\Psi\left(f_{0}\right)+\frac{W_n(a)}{\sqrt n}.
\end{align*}

The following Theorem characterizes the BvM property when $K$ is sufficiently large and when $\alpha$ is known.

\begin{theorem}\label{thm1}
Assume the model \eqref{model} with $p=1$, where $f_0$ is endowed with a prior on $\mF[K]$ in \eqref{prior1}  induced by \eqref{eq:priorb}.
Assume $f_0\in\Ha_1$ and $a\in\Haa_1$ for $1/2<\alpha\leq 1$ and $\gamma>1/2$.  With the choice $K=K_n=\lfloor (n/\log n)^{1/(2\alpha+1)}\rfloor$, we have 
$$
\Pi\left(\sqrt{n}\left(\Psi(f^K)-\wh\Psi_K\right)\C\Y^{(n)}\right)\rightsquigarrow \mathcal{N}(0,\|a\|_L^2)
$$
in $\P_0^n$-probability as $n\rightarrow\infty$.
\end{theorem}
\proof Section \ref{sec:proof:thm1}

\smallskip
\begin{remark}
Theorem \ref{thm1} applies to the so-called {\sl symmetric dyadic trees}.  In particular, when $n=2^r$ for some $r>0$, the equivalent blocks partition with $K=2^L$ cells  can be represented with a symmetric full binary tree which splits all the nodes at dyadic rationals up to a resolution $L$.
\end{remark}

Theorem \ref{thm1} is related to Theorem 4.2 of \cite{castillo_rous} for density estimation with non-adaptive histogram priors. The proof here also requires two key ingredients. The first one  is the construction of a prior which does not change too much under the change of measures from $f^K$ to  $f_t^K$, where $f_t^K$ is a step function  with shifted step sizes $\beta_k^t=\beta_k-\frac{ta_k^K}{\sqrt{n}}$.
This property holds for (correlated) Gaussian step heights and is safely satisfied by other prior distributions with heavier tails.
\begin{remark}
Theorem \ref{thm1} holds for Laplace product prior $\pi(\b\C K)=\prod_{k=1}^K\psi(\beta_k,\lambda)$ (as shown in Section \ref{sec:laplace_proof}). It is interesting to note that under the Laplace prior, one can obviate the need for showing posterior concentration around a projection of $f_0$ onto $\mF[K]$, which is needed for the Gaussian case.
\end{remark}
The second crucial ingredient (as in the Proposition 1 of \cite{castillo_rous}) is the so-called {\sl no bias condition}:
\begin{equation}\label{eq:nobias}
\sqrt{n}\langle a-a^K,f_0-f_0^K\rangle_L=o(1).
\end{equation}
This condition vaguely reads as follows: one should be able to approximate {\sl simultaneously} $a(\cdot)$ and $f_0(\cdot)$  well enough using  functions in the inferential class $\mF[K]$.   Using the Cauchy-Schwartz inequality and Lemma 3.2 of \cite{rockova_vdp}, this condition will be satisfied when $\sqrt{n}K^{-(\alpha+\gamma)}\rightarrow 0$. Choosing $K=K_n=\lfloor (n/\log n)^{1/(2\alpha+1)}\rfloor$, \eqref{eq:nobias} holds as long as $\gamma>1/2$.  Different choices of $K_n$, however, would imply different restrictions on $\alpha$ and $\gamma$.  
The no-bias condition thus enforces certain  limitations  on  the regularities $\alpha$ and $\gamma$.
This poses challenges for adaptive priors  that only adapt to the regularity of $f_0$, which may not be necessarily  similar to the regularity of $a$. This phenomenon has been coined as the {\sl curse of adaptivity} \citep{castillo_rous}.

\section{Overcoming the Curse of Adaptivity}\label{sec:semi:adaptive}
The dependence of $K_n$ on $\alpha$ makes the result in Theorem \ref{thm1}  somewhat theoretical. In practice, it is common to estimate the regularity from the data  using, e.g., a hierarchical Bayes method, which  treats both $K$ and the partition $\mT$  as unknown with a prior. This fully Bayes model brings us  a step closer to the actual Bayesian CART and BART priors.
Treating both $K$ and $\mT$ as random and assuming $T=1$, the class of approximating functions now  constitutes a union of shells
$$
 \mF=\bigcup_{K=1}^\infty \mF[K],
$$
where each shell  
$$
\mF[K]= \bigcup_{\mT\in\mV^K}\mF[\mT],
$$ 
itself is a union of sets $\mF[\mT]=\{f_{\mT,\b} \quad\text{of the form \eqref{tree} for some $\b\in\R^{K}$}\}$. The sets $\mF[\mT]$ collect all step functions that grow on the same  tree partition $\mT\in\mV^K$.

As mentioned in \cite{castillo_rous}, obtaining BvM in the case of random $K$ is case dependent. 
As the prior typically adapts to the regularity of $f_0$, the no-bias condition \eqref{eq:nobias} may not be satisfied if the regularities of $a$ and $f_0$ are too different. The adaptive prior can be detrimental in such scenarios,  inducing a non-vanishing bias  in the centering of the posterior distribution (see \cite{castillo_bias} or \cite{riv_rous}). Roughly speaking, one needs to make sure that the prior supports large enough $K$ values and sufficiently regular partitions $\mT$ so that $f_0$ and $a$ can be both safely approximated.
To ensure this behavior,  we enforce a signal strength assumption  through self-similarity requiring that the function $f_0$ ``does not appear smoother than it actually is" \citep{gine_nickl}. Such qualitative assumptions are natural and  necessary for obtaining  adaptation properties  of confidence bands  \citep{picard,bull, gine_nickl2,nickl_szabo,ray}.

\subsection{Self-similarity}\label{sec:selfsim}
Various self-similarity conditions have been considered  in various estimation settings, including the supremum-norm loss  \citep{bull, gine_nickl2,hoffman_nickl} as well as the $\ell_2$ loss 
\citep{nickl_szabo, szabo_etal}. In the multi-scale analysis, the term self-similar coins ``desirable truths" $f_0$ that are not so difficult to estimate since their regularity looks similar at small and large scales. \cite{gine_nickl2} argue that such self-similarity is a reasonable modeling assumption and  \cite{bull} shows the set of self-dissimilar  H\"{o}lder functions  (in the $\ell_\infty$ sense) is negligible. \cite{szabo_etal} provided an $\ell_2$-style self-similarity restriction on a Sobolev parameter space. \cite{nickl_szabo} weakened this condition and showed that it ``cannot be improved upon"  and that the statistical complexity of the estimation problem does not decrease quantitatively under self-similarity  in Sobolev spaces.

We consider a related notion of $\ell_2$ self-similar classes within the context of fixed-design regression as opposed to the  asymptotically equivalent white noise model.
To this end, let us first  formalize the notion of the cell size in terms of the local  spread of the data and introduce the partition {\sl diameter} \citep{verma2009,rockova_vdp}.
 


\begin{definition}(Diameter)
Denote by $\mT=\{\Omega_k\}_{k=1}^K$  a partition of $[0,1]^p$ and by $\mX=\{\x_1,\dots,\x_n\}$ a collection of data points in $[0,1]^p$.  We define a {\sl diameter} of $\Omega_k$  as 
$$
\mathrm{diam}\left(\Omega_k\right)=\max\limits_{\substack{\x,\y\in\Omega_k\cap\mX}}\|\x- \y\|_2.
$$
and with
$
\mathrm{diam}\left(\mT\right)=\sqrt{\sum_{k=1}^K\mu(\Omega_k)\mathrm{diam}^2\left(\Omega_k\right)}
$
we define a {\sl  diameter} of the entire partition $\mT$ where $\mu(\Omega_k)=\frac{1}{n}\sum_{i=1}^n\mathbb{I}(\x_i\in\Omega_k)=n(\Omega_k)/n$.
\end{definition}

Here, we do not require that the design points are strictly  on a grid as long as they are regular according to Definition 3.3 of \cite{rockova_vdp}.
We define regular datasets  below.  First we introduce the notion of the $k$-$d$ tree \citep{bentley1975}. 
Such a partition $\wh{\mT}$ is constructed by cycling over coordinate directions in $\mS=\{1,\dots,p\}$, where all nodes at the same level are split along the same axis. For a given direction $j\in\mS$, each  internal node, say $\Omega_k^\star$, will be split  at a median of the point set (along the $j^{th}$ axis). This split will pass $\lfloor \mu(\Omega_k^\star) n/2\rfloor$ and $\lceil \mu(\Omega_k^\star) n/2\rceil$ observations onto its two children, thereby roughly halving the number of points. 
After $s$ rounds of splits on each variable, all  $K$ terminal nodes have at least $\lfloor n/K \rfloor$ observations, where $K=2^{s\, |\mS|}$. 

We now define regular datasets in terms of the $k$-$d$ tree partition.

\begin{definition}\label{def:regular}
Denote by $\wh{\mT}=\{\wh{\Omega}_k\}_{k=1}^K\in\mV^K$ the $k$-$d$ tree where   $K=2^{s\, p}$. We say that a dataset $\mX\equiv\{\x_i\}_{i=1}^n$ is  regular if
\begin{equation}\label{eq:design_regular}
\max\limits_{1\leq k\leq K}\mathrm{diam}(\wh{\Omega}_k)<M\, \sum_{k=1}^K\mu(\wh{\Omega}_k)\mathrm{diam}(\wh{\Omega}_k)
\end{equation}
for some large enough constant $M>0$ and all  $s\in\N\backslash\{0\}$.
\end{definition}
The definition states  that in a regular dataset,  the maximal diameter  in the $k$-$d$ tree partition should not be much larger than a ``typical" diameter.


\begin{definition}
We say that the function $f_0\in\Ha_p$ is {\sl self-similar}, if  there exists constant $M>0$ and $D>0$ such that
\begin{equation}\label{selfsim2}
\|{f}^0_{\mT}-f_0\|_L^2\geq M \mathrm{diam}(\mT)^{2\alpha}\quad\text{for all}\quad \mT\in\mV^K\quad\text{such that}\quad \mathrm{diam}(\mT)\leq D ,
\end{equation}
where ${f}^0_{\mT}$ is the the $\|\cdot\|_L$  projection of $f_0$ onto $\mF[\mT].$
\end{definition}

We can relate the assumption \eqref{selfsim2}  to  the notion of self-similarity in  the Remark 3.4 of \cite{szabo_etal}. To see this connection, assume for now the equivalent block partition $\mT$  from Section \ref{sec:semi}, whose diameter  $\mathrm{diam}(\mT)$ is roughly $1/K$ when the design points lie on a regular grid. 
The study of regression histograms with $K=2^L$ equivalent blocks under a fixed regular design is statistically equivalent to the multi-scale analysis of Haar wavelet coefficients up to the resolution $L-1$ in the white noise model. The projected model onto the Haar basis can be written as 
 $$
 Y_{lk}=f_{lk}^0+\frac{1}{\sqrt{n}}\varepsilon_{lk}\quad\text{for}\quad 0\leq l< L\quad\text{and}\quad 0\leq k<2^l,
 $$ where $\varepsilon_{lk}\iid\mathcal{N}(0,1)$  and where $f_{lk}^0$ are the wavelet coefficients  indexed by the  resolution  level $l$ and the shift index $k$.
 The speed of decay of $f_{lk}^0$ determines the statistical properties of $f_0$, where $\alpha$-H\"{o}lder continuous functions satisfy $|f_{lk}^0|\lesssim 2^{-l(\alpha+1/2)}$. Assuming the equivalent blocks partition $\mT$ with $K=2^L$ blocks,  the condition in the Remark 3.4 of  \cite{szabo_etal} writes as follows: there exists $K_0\in\N$ such that $\forall K\geq K_0$ we have
 $$
 \int_0^1|{f}^0_{\mT}(x)-f_0(x)|^2\d x=\sum\limits_{l\geq L}\sum\limits_{0\leq k<2^l}(f_{lk}^0)^2\geq M K^{-2\alpha}\asymp \mathrm{diam}(\mT)^{2\alpha}.
 $$ 
 The first equality above  stems from the orthogonality of the Haar bases. In this vein, \eqref{selfsim2} can be regarded a generalization of this condition to imbalanced partitions and fixed design under the norm $\|\cdot\|_L$.

To get even more insight into \eqref{selfsim2} in  fixed design regression, we take a closer look at the approximation gap.
We have 
$$
\|f_0-f^{0}_\mT\|_L^2
={\sum_{k=1}^{K}\mu(\Omega_{k})\Var[f_0\C\Omega_{k}]},
$$
where
$$
\Var[f_0\C\Omega_k]=\frac{1}{n(\Omega_k)}\sum_{\x_i\in\Omega_k}\left(f_0(\x_i)-\frac{1}{n(\Omega_k)}\sum_{\x_i\in\Omega_k}f_0(\x_i)\right)^2
$$
is the local variability of the function $f_0$ inside $\Omega_k$.  The function $f_0$ will be self-similar when the variability  inside each cell $\Omega_k$'s is large enough to be detectable, i.e.
$$
\inf_{1\leq k\leq K}\Var[f_0\C\Omega_k]>M\mathrm{diam}^{2\alpha}(\mT)\quad\text{for all}\quad  \mT=\{\Omega_k\}_{k=1}^K\in\mV^K
$$
such that $\mathrm{diam}(\mT)\leq D$ for some $D>0$. From the definition of the partition diameter, it turns out that this will be satisfied when $\Var[f_0\C\Omega_k]>\mathrm{diam}^{2\alpha}(\Omega_k)$ for all $1\leq k\leq K$ and $\mT=\{\Omega_k\}_{k=1}^K\in\mV^K.$
Functions that are nearly constant in long intervals will not satisfy this requirement.
The premise of self-similar functions is that their signal should be detectable with partitions $\mT$ that undersmooth the truth. 

In addition, it follows from the proof of Lemma 3.2 of \cite{rockova_vdp} that $\|{f}^0_{\mT}-f_0\|_L^2\lesssim \mathrm{diam}(\mT)^{2\alpha}$.
 The lower bound in \eqref{selfsim2} thus matches the  upper bound making the approximation  error behave similarly across partitions with similar diameters, essentially identifying the smoothness.

Based on these considerations, we introduce the notion of {\sl regular}  partitions that are not too rough in the sense that their diameters shrink sufficiently fast.
\begin{definition}
For some $M>0$ and for some arbitrarily slowly increasing sequence $M_n\rightarrow\infty$, we denote 
\begin{equation}\label{k_n}
d_n(\alpha)\equiv (M_n/M)^{1/\alpha}{n^{-1/(2\alpha+p)}}{\log^{1/2\alpha}n}.
\end{equation}
A tree partition $\mT\in\mV^K$ is said to be $n$-{\sl regular} for a given $n\in\N$  if 
$$
\mathrm{diam}(\mT)\leq d_n(\alpha).
$$
We denote the subset of all $n$-regular partitions with $\mR_n$.
\end{definition}
The following Lemma states that,
when $f_0$ is self-similar, the posterior concentrates on partitions that are not too complex or irregular.

\begin{lemma}\label{lemma1}
Assume  that $f_0\in\mathcal{H}_p^\alpha$ is self-similar,  $p\lesssim \sqrt{\log n}$ and that the design $\mX\equiv\{\x_i\}_{i=1}^n$ is regular. Under the Bayesian CART prior (\eqref{prior:K} and \eqref{prior:partition} or \eqref{eq:split_prob}) with Gaussian or Laplace step heights (\eqref{eq:priorb} or \eqref{eq:priorb_laplace}) we have
$$
\Pi\left(\{\mT\notin \mR_n\}\cup\{\mT\in\mV^K:K> K_n \}\C\Y^{(n)}\right)\rightarrow 0
$$
in $\P_0^n$-probability as $n,p\rightarrow\infty$, where $\mR_n$ are all regular partitions and $K_n=M_2\lfloor n\varepsilon_n^2/\log n\rfloor\asymp (n/\log n)^{p/(2\alpha+p)}$ for some $M_2>0$.
\end{lemma}
\begin{proof}
With $\varepsilon_n=n^{-\alpha/(2\alpha+p)}\sqrt{\log n}$, the assumption
\eqref{selfsim2} implies
$
\|{f}_{0}^{\mT}-f_0\|_L> M_n/M_3\,\varepsilon_n
$
when $\mathrm{diam}(\mT)>d_n(\alpha)$, where $f_0^{\mT}$ is the $\|\cdot\|_L$ projection of $f_0$ onto $\mF[\mT]$. The posterior distribution under the Bayesian CART prior concentrates at the rate $\varepsilon_n$ in the $\|\cdot\|_L$ sense, i.e. $\Pi(\|f-f_0\|_L>M_n\varepsilon_n\C\Y^{(n)})\rightarrow 0$ in $\P^n_0$-probability for any arbitrarily slowly increasing sequence $M_n$.
This follows  from \cite{rockova_vdp} for the conditionally uniform tree partition prior and from \cite{rockova_saha} for the tree-branching process prior. Both of these papers study Gaussian step heights. In the Appendix (Section \ref{sec:append_laplace}), we extend these results to Laplace step heights. From the near-minimaxity of the posterior, it then follows that    partitions that are {\sl not} regular  are  {\sl not} supported  by the posterior.  The dimensionality part regarding $K$ follows from \cite{rockova_vdp} and \cite{rockova_saha}.
\end{proof}

\subsection{Adaptive BvM for Smooth Functionals when $p=1$}\label{sec:adaptive}
It is known that signal-strength conditions enforced through self-similarity allow for the construction of honest adaptive credible balls \citep{gine_nickl2}.
Our notion of self-similarity is sufficient for obtaining  the adaptive semi-parametric BvM phenomenon for smooth linear functionals.
Denote with 
$$
\mR(K_n)=\{\mT\in \mR_n\cap \mV^K\quad\text{for}\quad K\leq K_n \}.
$$ 
Lemma \ref{lemma1} shows that the posterior concentrates on this set so that we can perform the analysis locally.
For any $\mT\in\mR(K_n)$, we write
$$
\wh\Psi_{\mT}=\Psi(f_0^{\mT})+\frac{W_n(a_{\mT})}{\sqrt{n}},
$$
where $f_0^{\mT}$ and $a_{\mT}$ are the $\|\cdot\|_L$ projections onto $\mF[\mT]$. 
Under the adaptive prior (treating the partitions as random with a prior) the posterior is asymptotically close to a {\sl mixture} of normals indexed by $\mT\in\mR(K_n)$ with weights $\pi(\mT\C\Y^{(n)},\mR(K_n))$. When 
\begin{equation}\label{eq:conditions}
\max_{\mT\in\mR(K_n)}\big|\|a_\mT\|_L-\|a\|_L\big|=o(1)\quad\text{and}\quad \max_{\mT\in\mR(K_n)}\sqrt{n}(\Psi_n-\wh\Psi_\mT)=o_P(1),
\end{equation}
this mixture boils down to the target law 
$\mathcal{N}(0,\|a\|_L^2)$. The first  condition in \eqref{eq:conditions}  holds  owing to the fact that $\|a-a_\mT\|_L\lesssim d_n(\alpha)^\gamma\rightarrow 0$ (Lemma 3.2 of Rockova and van der Pas (2017)).
The second condition will be satisfied as long as 
\begin{equation}\label{diam}
\sqrt{n}\, d_n(\alpha)^{\alpha+\gamma}\rightarrow 0.
\end{equation}
For $\mT\in\mR(K_n)$ and assuming $p=1$, we have for $M_n\lesssim \sqrt{\log n}$
$$
\sqrt{n}\, d_n(\alpha)^{\alpha+\gamma}\lesssim n^{1/2-\frac{\alpha+\gamma}{2\alpha+1}}(\log n)^{\frac{\alpha+\gamma}{\alpha}}
\rightarrow 0
$$
for $\gamma>1/2$. We can now state an adaptive variant of Theorem \ref{thm1} for random partitions.


\begin{theorem}\label{thm3}
Assume model \eqref{model} with $p=1$, where  $f_0\in\Ha_1$ and $a\in\Haa_1$ for $1/2<\alpha\leq 1$ and $\gamma>1/2$. 
Assume that $f_0$ is {\sl self-similar}. Under the Bayesian CART prior  (\eqref{prior:K} and \eqref{prior:partition} or \eqref{eq:split_prob}) with Gaussian or Laplace step heights (\eqref{eq:priorb}  or \eqref{eq:priorb_laplace}), we have
$$
\Pi\left(\sqrt{n}\left(\Psi(f_{\mT,\b})-\wh\Psi_{\mT}\right)\C\Y^{(n)}\right)\rightsquigarrow \mathcal{N}(0,\|a\|_L^2)
$$
in $\P_0^n$-probability as $n\rightarrow\infty$.
\end{theorem}

\begin{proof}
Section \ref{sec:proof_thm3}
\end{proof}

Theorem \ref{thm3} can be regarded as an adaptive extension of the regular density histogram result of \cite{castillo_rous}. Here, we instead focus on irregular and adaptive regression histograms underpinned by tree priors and treat both  $K$  and  $\mT$ as random under self-similarity. The change of measure argument is performed locally for each regular partition.

Theorem \ref{thm3} can be extended to tree ensembles. The self-similarity assumption would be instead formulated in terms of a {\sl global partition}, which is obtained by super-imposing all tree partitions inside $\mE$ and by merging empty cells with one of their neighbors.   Since tree ensembles also concentrate at near-minimax speed \citep{rockova_vdp, rockova_saha}, one obtains that the posterior concentrates on regular ensembles (where the diameter is small). The analysis is then performed locally on regular ensembles in the same spirit as for single trees.

\subsection{BvM for Average Regression Surface when $p>1$}

One of the main limitations of tree/forest methods is that they cannot estimate optimally functions that are smoother than Lipschitz \citep{scricciolo}. 
 The reason for this limitation is that step functions are relatively rough; e.g. the approximation error of histograms for functions that are smoother than Lipschitz is at least of the order $1/K$, where $K$ is the number of bins \citep{rockova_vdp}. The number of steps required to approximate a smooth function well is thus too large, creating a costly bias-variance tradeoff.
When  $p>1$, the no-bias condition \eqref{eq:nobias} would be satisfied if  $\gamma>p/2$ which, from the H\"{o}lder continuity,  holds when $a(\x_i)$ is a constant function. 

Focusing on the actual BART method when $p>1$,  we now rephrase Theorem \ref{thm3} for the average regression functional \eqref{psi_functional} obtained with $a(\cdot)=1$.
When $a(\cdot)$ is a constant function, the no-bias condition \eqref{eq:nobias} is automatically satisfied. Recall that the second requirement for BvM pertains to  the shift of measures.  It turns out that the Gaussian prior \  \eqref{eq:priorb} may induce too much bias when the variance is too small (fixed as $n\rightarrow \infty$). We thereby deploy an additional assumption in the prior to make sure that the variance increases suitably with the number of steps $K$.
For the BART prior on step heights $\b^t$ of each tree $\mT^t\in\mE$, we assume either a Gaussian prior
$\b^t\sim\mathcal{N}_{K^t}(\bm 0,K^t\times I_{K^t})$ or the Laplace prior with $\lambda_t\asymp 1/\sqrt{K^t}$.  
 
Having the variance scale with the number of steps is generic in the multi-scale analysis of  Haar wavelets. For instance, assuming a full dyadic tree with $L-1$ layers, where $L=\log_2 K$, the iid standard Gaussian product prior on the wavelet coefficients implies a Gaussian prior with variance $K$ on the bottom nodes \citep{castillo_rockova}. The bottom nodes correspond to histogram bins in our setup, where more refined partitions will ultimately have more variable step sizes.

The following theorem is formulated for a few variants of the BART prior. This prior consists of (a) either fixed number of trees $T$ (as recommended by \citep{bart}), (b) the Galton-Watson prior \eqref{eq:split_prob} or the conditionally uniform tree prior \eqref{prior:K} and \eqref{prior:partition}, independently for each tree, and (c) the Gaussian prior $\b^t\sim\mathcal{N}_{K^t}(\bm 0,K^t\times I_{K^t})$ or the Laplace prior with $\lambda_t\asymp 1/\sqrt{K^t},$ where $K^t$ is the number of bottom nodes of a tree $\mT^t$. Below, we denote with 
$\wh\Psi_{\mE}=\Psi(f^0_\mE)+W_n(a)/\sqrt{n}$, where $f^0_\mE$ is a projection of $f_0$ onto $\mF[\mE],$ a set of all forest mappings \eqref{forest} supported on the tree ensemble $\mE$.
 
\begin{theorem}\label{thm4}
Assume model \eqref{model} with $p\geq 1$, where $f_0\in\Ha_p$ is endowed with the BART prior (as stated above) and where $\log p\lesssim n$  and $1/2\leq \alpha<1$. Assume that 
 $a(\cdot)=1$ in \eqref{psi_functional}. When the design $\mX=\{\x_i\}_{i=1}^n$ is regular, we have
$
\Pi\left(\sqrt{n}\left(\Psi(f_{\mE,\mB})-\wh\Psi_{\mE}\right)\C\Y^{(n)}\right)\rightsquigarrow \mathcal{N}(0,\|a\|_L^2)
$
in $\P_0^n$-probability as $n\rightarrow\infty$.
\end{theorem}

\begin{proof}
Section \ref{sec:thm4_proof}
\end{proof}

\section{Discussion}
This note reveals some aspects of  Bernstein-von Mises limits under adaptive BART priors. We focus on semi-parametric BvM's for linear functionals of the infinite-dimensional regression function parameters. This semi-parametric setup already poses nontrivial challenges on hierarchical Bayes. We have reiterated and highlighted some of the challenges here and addressed them with self-similarity identification.  Our results serve as a step towards obtaining fully non-parametric BvM for the actual function $f_0$, as opposed to just its low-dimensional summaries. These results will be reported in follow-up work.

\section{Appendix}

\subsection{Proof of Theorem \ref{thm1}}\label{sec:proof:thm1}

First, we show that with $K$ large enough $\wh\Psi_K$  will satisfy
\begin{equation}\label{psis}
\sqrt{n}|\Psi_n-\wh\Psi_K|=o_P(1).
\end{equation}
We can write
\begin{align*}
&\sqrt{n}|\Psi_n-\wh\Psi_K|=\sqrt{n}|\Psi(f_0)-\Psi(f_0^K)+W_n(a-a^K)|\\
&\quad\quad=\sqrt{n}\left|\frac{1}{n}\sum_{i=1}^n \left[f_0(\x_i)-f^K_0(\x_i)+\varepsilon_i\right]\left[a(\x_i)-a^K(\x_i)\right] \right|,
\end{align*}
where we used the fact that
$$
\sum_{i=1}^n[f_0(\x_i)-f^K_0(\x_i)]a^{K}(\x_i)=\sum_{k=1}^K\frac{a_k^K}{n}\sum_{\x_i\in\Omega_k}\left[f_0(\x_i)-\frac{n}{\mu(\Omega_k)}\sum_{\x_i\in\Omega_k}f_0(\x_i)\right]=0.
$$
Next, we can write
\begin{align*}
&\sqrt{n}|\Psi_n-\wh\Psi_K|<\left[\sqrt{n}\|a-a^K\|_L\times\|f_0-f_0^K\|_L+Z_n^K\right],
\end{align*}
where 
$$
Z_n^K=\frac{1}{\sqrt{n}}\sum_{i=1}^n\varepsilon_i[a(\x_i)-a^K(\x_i)].
$$
We assume that the design is regular (according to Definition 3.3 of Rockova and van der Pas (2017) (further referred to as RP17) with $p=q=1$).
Assuming that $a\in\Haa$,
it follows from the proof of Lemma 3.2 of Rockova and van der Pas (2017) that
$$
\|f_0-f_0^K\|_L\leq \|f_0\|_{\Ha}C_1/K^{\alpha}\quad\text{and}\quad \big|\|a\|_L-\|a^K\|_L\big|\leq \|a-a^K\|_L\leq \|a\|_{\Haa}C_2/K^{\gamma}.
$$
We assume that $\|a\|_L^2\leq \|a\|_\infty^2<C_a^2$ for some $C_a>0$ and $\|f_0\|_\infty<C_f$ for some $C_f>0$. 
Because
$$
\Var Z_{n}^K=\|a-a^K\|_L^2=\frac{1}{n}\sum_{i=1}^n[a(\x_i)-a^K(\x_i)]^2\lesssim 1/K^{2\gamma},
$$
we conclude that $Z_n^K=o_P(1)$ when $K\rightarrow\infty$ as $n\rightarrow\infty$ and thereby
\begin{align}
&\sqrt{n}|\Psi_n-\wh\Psi_K|\lesssim\sqrt{n}K^{-(\alpha+\gamma)} +o_P(1).\label{eq:nobias1}
\end{align}
With the choice $K=K_n=\lfloor(n/\log n)^{1/(2\alpha+1)}\rfloor$  for $\alpha>1/2$ and $\gamma>1/2$, \eqref{eq:nobias1} will be satisfied.

To continue, we introduce the following notation
\begin{align*}
f_t^K=f^K-\frac{t\,a^K}{\sqrt n}
\end{align*}
and write
\begin{align*}
&\ell_n(f_t^K)-\ell_n(f_0^K)-[\ell_n(f^K)-\ell_n(f_0^K)]\\
&\quad\quad=-\frac{n}{2}[\|f_t^K-f_0^K\|_L^2-\|f^K-f_0^K\|_L^2]+\sqrt{n}\,W_n(f_t^K-f^K)\\
&\quad\quad=-\frac{n}{2}[\|f_t^K-f^K\|_L^2+2\langle f_t^K-f^K,f^K-f_0^K\rangle_L]+\sqrt{n}\,W_n(f_t^K-f^K)\\
&\quad\quad= -\frac{t^2}{2}\|a^K\|_L^2+\sqrt{n}\,t\langle a^K,f^K-f_0^K\rangle_L-t\, W_n(a^K)
\end{align*}
Then we have
$$
t\sqrt{n}(\Psi(f^K)-\wh\Psi_K)=t\sqrt{n}(\Psi(f^K)-\Psi(f_0^K))-t\,W_n(a^K)
$$
and thereby, using the fact 
$$
t\sqrt{n}(\Psi(f^K)-\Psi(f_0^K))=t\sqrt{n}\langle a^K,f^K-f_0^K\rangle_L,
$$
we can write
\begin{align*}
&t\,\sqrt{n}(\Psi(f^K)-\wh\Psi_K)+\ell_n(f^K)-\ell_n(f_0^K)\\
&\quad\quad=\ell_n(f_t^K)-\ell_n(f_0^K)+\frac{t^2}{2}\|a^K\|_L^2-\sqrt{n}\,t\langle a^K,f^K-f_0^K\rangle_L+t\sqrt{n}\langle a^K,f^K-f_0^K\rangle_L\\
&\quad\quad=\ell_n(f_t^K)-\ell_n(f_0^K)+\frac{t^2}{2}\|a^K\|_L^2.
\end{align*}
Given sets $A_{n,K}\subset \mF[K]$  (to be defined later) such that $\Pi(A_{n,K}^C\C\Y^{(n)})\rightarrow 0$ in $\P_0^n$ probability,
we define
\begin{equation}\label{Ink}
I_{n,K}=\E^{\Pi}[\e^{t\sqrt{n}(\Psi(f^K)-\wh{\Psi}_K)}\C\Y^{(n)}, A_{n,K}]
\end{equation}
and, using the calculations above, we write
\begin{equation}\label{ink}
I_{n,K}=\e^{\frac{t^2}{2}\|a^K\|_L^2}\times\frac{\int_{A_{n,K}}\e^{\ell_n(f_t^K)-\ell_n(f_0^K)}\d\Pi^K(f^K)}{\int_{A_{n,K}}\e^{\ell_n(f^K)-\ell_n(f_0^K)}\d\Pi^K(f^K)}.
\end{equation}
For $\b=(\beta_1,\dots,\beta_K)'\in\R^K$, define $\beta_k^t=\beta_k-\frac{t a_k^K}{\sqrt n}$ and denote $\b^t=(\beta_1^t,\dots,\beta_K^t)\in\R^K$ and $\a^K=(a_1^K,\dots, a_K^K)'\in\R^K$. Then we have
\begin{align*}
f^K(\x)=\sum_{k=1}^K\1(\x\in\Omega_k)\beta_k\quad\text{and}\quad f^K_t(\x)&=\sum_{k=1}^K\1(\x\in\Omega_k)\beta_k^t.
\end{align*}
Assuming the multivariate Gaussian prior $\pi(\b)=\frac{1}{\sqrt{2\pi|\Sigma|}}\e^{-\frac{1}{2}\b'\Sigma^{-1}\b}$ 
centered at zero with a covariance matrix $\Sigma$, we can write
\begin{align}\label{change}
\pi(\b)&=\pi(\b^t)\e^{\frac{t^2}{2n}\a^{K'}\Sigma^{-1}\a^K- \frac{t}{\sqrt n}\, \a^{K'}\Sigma^{-1}\b}.
\end{align} 
Next, for $\wt\varepsilon_{n,K}=\sqrt{\frac{K\,\log (n)}{n}}$ and $M>0$, we define
\begin{equation}\label{def:Ank}
A_{n,K}(M)=\left\{f^K\in\mF[K]: \|f^K-f_0^K\|_L\leq M\,\wt\varepsilon_{n,K}\right\}.
\end{equation}
We show (Section \ref{sec:setA}) that $\Pi(A_{n,K}^C(M)\C\Y^{(n)})\rightarrow 0$ in $\P_0^n$-probability as $n\rightarrow\infty$ for some suitably large $M>0$.
We note that
$$
A_{n,K}(M)\subset\{\b\in\R^K:\|\b-\b_0^K\|^2_2\leq M^2\,n\,\wt\varepsilon_{n,K}^2 \},
$$
where $\b_0^K\in\R^K$ are coefficients of the projection of $f_0$ onto $\mF[K]$. On the set $A_{n,K}(M)$, we can thus write
\begin{align}
\frac{t}{\sqrt n}\, \big|\a^{K'}\Sigma^{-1}\b\big|&<\frac{t}{\sqrt n\,\lambda_{min}}\|a^K\|_2\sqrt{\|\b-\b^K_0\|_2^2+\|\b^K_0\|_2^2}\notag\\
&<\frac{t}{\sqrt n\,\lambda_{min}}\|a^K\|_2\sqrt{K\log n+KC_f^2}\notag\\
&<\frac{t\,K\,C_a}{\sqrt n\,\lambda_{min}}\sqrt{\log n+C_f^2},\label{schwartz}
\end{align}
where $\lambda_{min}$ is the smallest singular value of $\Sigma$ and 
where we used the fact that  $\|a\|_\infty<C_a$ and $\|f_0\|_\infty<C_f$. Using \eqref{schwartz} and \eqref{change}, we have
\begin{align*}
\e^{\frac{t^2}{2}\|a\|_L^2 
-\frac{t\,K\,C_a}{\lambda_{min}\,\sqrt n}\sqrt{\log n+C_f^2}}<I_{n,K}&<\e^{\frac{t^2}{2}\|a\|_L^2+\frac{t^2K\,C_a^2}{2\,n\,\lambda_{min}}+\frac{t\,K\,C_a}{\lambda_{min}\,\sqrt n}\sqrt{\log n+C_f^2}}.
\end{align*}
If $\lambda_{min}>c$ for some $c$ and $K\sqrt{(\log n)/n} \rightarrow 0$, we have
$\lim_{n\rightarrow\infty}I_{n,K}=\e^{t^2\|a\|^2_L}$ for each $t\in\R$. This is satisfied if we set, for instance, $K=K_n=\lfloor \left(n/\log n\right)^{1/(2\alpha+1)}\rfloor$  with $\alpha>1/2$.
This concludes the proof of the BvM property for a  fixed $K$ and a single partition.

\subsubsection{The Laplace Prior}\label{sec:laplace_proof}
 For the Laplace prior, we use the reverse triangle inequality $|\beta_k|>|\beta_k^t|-|\frac{t a_k^K}{\sqrt n}|$ to find that
 \begin{align*}
 \pi(\b^K)&=\prod_{k=1}^K\psi(\beta_k;\lambda)<\prod_{k=1}^K\psi(\beta_k^t;\lambda)\e^{\frac{t\lambda |a_k^K|}{\sqrt n}}=\pi(\b^K_t)\e^{ \frac{t\lambda}{\sqrt n}\,\|\a^K\|_1}<\pi(\b^K_t)\e^{ \frac{t\lambda}{\sqrt n}\, KC_a}.
\end{align*}
Then we find that for $K=K_n$
$$
\e^{\frac{t^2}{2}(\|a\|_L^2+o(1))}\times  \e^{-\frac{t\lambda}{\sqrt n}\,K_n(\alpha)C_a}<I_{n,K}<\e^{\frac{t^2}{2}(\|a\|_L^2+o(1))}\times  \e^{\frac{t\lambda}{\sqrt n}\,K_n(\alpha)C_a}.
$$
Since $K_n(\alpha)=\lfloor (n/\log n)^{1/(2\alpha+1)}\rfloor$, we have for $\alpha>1/2$ and $t\in\R$. 
\begin{align*}
\lim_{n\rightarrow\infty}\E^{\Pi}[\e^{t\sqrt{n}(\Psi(f^K)-\Psi_n)}\C\Y^{(n)}]=\e^{\frac{t^2}{2}\|a\|^2_L}. \quad\qedhere
\end{align*}
It is interesting to note that with the Laplace prior, one can obviate the proof of  posterior concentration around the projections, which was needed for the Gaussian case.

\subsubsection{The Set $A_{n,K}$}\label{sec:setA}
We want to show that the posterior distribution concentrates around $f_0^K$, the projection of $f_0$ onto $\mF[K]$, at the following contraction rate
$$
\wt\varepsilon_{n,K}= \sqrt{\frac{K\log(n)}{n}}.
$$ 
For $K\leq n/\log(n)$ and $A_{n,K}(M)$ defined in \eqref{def:Ank},
we show that
\begin{equation}\label{post_conc1}
\lim_{n\rightarrow\infty}\Pi(A_{n,K}^C(M)\C\Y^{(n)})=0\quad\text{in $\P_0^n$-probability}
\end{equation}
for some sufficiently constant  $M>0$. We show this statement by verifying conditions (2.4) and (2.5) of Theorem 2.1 of Kelijn and van der Vaart (2006).
We start with the entropy condition (2.5). In our model, the covering number for testing under misspecification can be bounded by the classical local entropy (according to Lemma 2.1 by Kelijn and van der Vaart (2006). It follows from Section 8.1 of Rockova and van der Pas (2017) that the local entropy satisfies
$$
N\Big(\tfrac{\varepsilon}{36}, \Big\{f^K \in \mF[K]: \|f^K - f_0^K\|_L < \varepsilon\Big\}, \|.\|_L\Big) \leq \left(\frac{108}{\bar{C}} \sqrt{n}\right)^K,
$$
where $\bar C$ is such that $\mu(\Omega_k)>\bar C/n$.
The entropy condition (2.5)  will be met since
$$
K\log n\lesssim n\,\wt\varepsilon_{n,K}^2.
$$
Regarding the prior concentration condition (2.4), we note  (similarly as in Section 8.2 of Rockova and van der Pas (2017)) that
$$
\left\{f^K\in\mF[K]:\|f^K-f_0^K\|_L\leq M\, \wt\varepsilon_{n,K}\right\}\supset\{\b\in\R^K:\|\b-\b_0^K\|_2\leq M\, \wt\varepsilon_{n,K}\}
$$
With the  Gaussian prior $\b\C K\sim\mathcal{N}_K(0,\Sigma)$, we have
\begin{align*}
\Pi\left(\b\in\R^{K}:\|\b-\b_0^K\|_2\leq M\, \wt\varepsilon_{n,K}\right)&
\geq \Pi\left(\wt \b\in\R^{K}:\|\wt\b-\wt\b_0^K\|_2\leq M\, \frac{\wt\varepsilon_{n,K}}{\sqrt{\lambda_{max}}}\right),
\end{align*}
where $\lambda_{max}$ is the maximal eigenvalue of $\Sigma$ and where $\wt\b_0^K=\Sigma^{-1/2}\b_0^K$ and $\wt\b=\Sigma^{-1/2}\b\sim\mathcal{N}_K(0,\mathrm{I}_K)$. The right-hand side can be further lower-bounded with
\begin{align*}
 \frac{2^{-K}\e^{-\|\wt\b_0^K\|_2^2-M^2 \wt\varepsilon_{n,K}^2/(4\lambda_{max})}}{\Gamma\left(\frac{K}{2}\right)\left(\frac{K}{2}\right)}\left(\frac{M \wt\varepsilon_{n,K}}{\sqrt{\lambda_{max}}}\right)^{K}.
\end{align*}
With $\lambda_{min}$ denoting the minimal eigenvalue of $\Sigma$, we obtain the following lower bound for the above:
$$
\e^{-C_1\, K\log \left(C_2\,K\,\lambda_{max}/\wt \varepsilon_{n,K}\right)-K\|f_0\|_\infty^2/\lambda_{min}-C_3 \wt\varepsilon_{n,K}^2/\lambda_{max}}
$$
With $\lambda_{min}>c$ for some $c>0$,  $\lambda_{max}\lesssim n$ and $\|f_0\|_{\infty}\leq \log^{1/2}(n)$, the above is bounded from below by $\e^{-D\, K\log(n)}$ for some suitable $D>0$.
It then follows from Theorem 2.1 of \cite{kleijn_vdv} that \eqref{post_conc1} holds.

\subsection{Posterior Concentration for the Laplace Prior}\label{sec:append_laplace}
\cite{rockova_vdp} and \cite{rockova_saha} show posterior concentration for BART under (a) the conditionally uniform prior \eqref{prior:K} and \eqref{prior:partition} and (b) the Galton Watson Process prior \eqref{eq:split_prob}. Both of these results apply for Gaussian step heights. Here, we formally show that the Laplace prior gives rise to optimal posterior concentration as well.

\begin{theorem}\label{thm:trees}
Assume $f_0\in \mathcal{H}^{\alpha}_{p}$  with {$0<\alpha\leq 1$} where
{$p \lesssim \log^{1/2} n$ and $\|f_0\|_\infty\lesssim\log^{1/2} n$}. Moreover, we assume that $\mX\equiv\{\x_i\}_{i=1}^n$ is regular. We assume priors \eqref{prior:K} and \eqref{prior:partition} or the Galton Watson process \eqref{eq:split_prob} and Laplace step heights \eqref{eq:priorb_laplace} where $1/\lambda\lesssim \sqrt{K}$.  Then with $\varepsilon_n=n^{-\alpha/(2\alpha+p)}\log^{1/2} n$ we have 
\begin{equation*}
\Pi\left( f_{\mT,\b} : \|f_{\mT,\b} - f_0\|_n > M_n\,\varepsilon_n \mid \Y^{(n)}\right) \to 0,
\end{equation*}
 for any $M_n \to \infty$  in $\P_{0}^n$-probability, as $n,p \to \infty$. 
\end{theorem}
\begin{proof}
 It suffices to show the prior concentration condition (2.2) in Rockova and van der Pas (2017), i.e.
$$
\Pi(f_{\mT,\b}:\|f_{\mT,\b}-f_0\|_n\leq \varepsilon_n)\geq \e^{-d\,n\,\varepsilon_n^2}
$$
for some $d>2$.  Using similar considerations as in Section 8.2. of RP17, this boils down to showing that
$$
\Pi(\b\in\R^K:\|\b-\wh\b\|_2\leq \varepsilon_n/2)\geq \Pi(\b\in\R^K:\|\b-\wh\b\|_1\leq \varepsilon_n/2),
$$
where $\wh\b\in\R^K$ are the step heights of the projection of $f_0$ onto a partition supported by the $k$-$d$ tree with $K$ steps, where $K\lesssim n\varepsilon_n^2/\log n$. The right hand side equals
\begin{align*}
\int_{|\b-\wh \b|_1\leq \varepsilon_n/2}\left(\frac{\lambda}{2}\right)^K\e^{-\lambda|\b|_1}\d\b&>\e^{-\lambda|\wh\b|_1}\int_{|\b|_1\leq\varepsilon_n/2}\left(\frac{\lambda}{2}\right)^K\e^{-\lambda|\b|_1}\d\b\\
&>\e^{-\lambda|\wh\b|_1}\left(\frac{\varepsilon_n\lambda}{2K}\right)^K\e^{-\lambda\varepsilon_n/2}.
\end{align*}
Assuming that $\|f_0\|_\infty\lesssim{\log n}$, we have $|\wh\b|_1\leq K\log n\lesssim n\varepsilon_n^2$. Next,   for $1/\lambda\lesssim \sqrt{K}$ 
we have $K\log[4K/(\lambda\varepsilon_n^2)]\lesssim K\log n\lesssim n\varepsilon_n^2$.
\end{proof}
\subsection{Proof of Theorem \ref{thm3}}\label{sec:proof_thm3}
According to Lemma \ref{lemma1}, the posterior concentrates on the set $\mR(\mK_n)$. All the following arguments will be thus conditional on $\mR(\mK_n)$.
The conditional posterior decomposes into a mixture of laws with weights  $\pi(\mT\C\Y^{(n)},\mR(\mK_n))$ in the sense that
\begin{align*}
&\E^{\Pi}[\e^{t\sqrt{n}(\Psi(f_{\mT,\b})-\Psi_n)}\C\Y^{(n)},\mR(\mK_n)]\\
&\quad\quad=\sum_{\mT\in\mR(\mK_n)}\pi[\mT\C \Y^{(n)},\mR(\mK_n)]\E^{\Pi}[\e^{t\sqrt{n}(\Psi(f_{\mT,\b})-\Psi_n)}\C\Y^{(n)}, \mT]\\
&\quad\quad=\e^{t\times o_P(1)}\sum_{\mT\in\mR(\mK_n)}\pi[\mT\C \Y^{(n)},\mR(\mK_n)]I_{n,\mT}
\end{align*}
where 
$$
I_{n,\mT}=\E^{\Pi}\left[\e^{t\sqrt{n}(\Psi(f_{\mT,\b})-\wh\Psi_{\mT})}\big|\Y^{(n)},\mT\right]
$$ 
and where we used the fact that $\sqrt{n}|\wh\Psi_\mT-\Psi_n|=o_P(1)$ under the assumption of self-similarity (as we showed earlier in the Section \ref{sec:adaptive}).
Under the  Laplace prior \eqref{eq:priorb_laplace},
we can write for $\BO\in\mR(\mK_n)$
$$
\e^{\frac{t^2}{2}(\|a\|_L^2+o(1))}\times  \e^{-\frac{t\,\lambda}{\sqrt n}\,K_n C_a}<I_{n,\BO}<\e^{\frac{t^2}{2}(\|a\|_L^2+o(1))}\times  \e^{\frac{t\,\lambda}{\sqrt n}\,K_n C_a}.
$$
Since $K_n=\lfloor M_2n^{1/(2\alpha+1)}\rfloor$, we have for $1/2<\alpha\leq 1$ and for all $t\in \R$
\begin{align*}
\lim_{n\rightarrow\infty}\E^{\Pi}[\e^{t\sqrt{n}(\Psi(f_{\BO,\b})-\Psi_n)}\C\Y^{(n)},\mR(\mK_n)]=\e^{\frac{t^2}{2}\|a\|^2_L}.\quad \qedhere
\end{align*}
For the Gaussian prior, one can proceed analogously as before. For each $\mT\in\mR(\mK_n)\cap\mV^K$, we denote with $A_{n,\mT}(M)=\{f_{\mT,\b}\in\mF[\mT]:\|f_{\mT,\b}-f_0^\mT\|_L\leq M\, \sqrt{K\log n/n}\}$. Using the same arguments as in Section \ref{sec:setA}, one can show that, given $\mT\in\mR(\mK_n)$, the posterior concentrates on $A_{n,\mT}(M)$. We then define 
$I_{n,\mT}=\E^{\Pi}\left[\e^{t\sqrt{n}(\Psi(f_{\mT,\b})-\wh\Psi_{\mT})}\big|\Y^{(n)},\mT,A_{n,\mT}(M)\right]$ and using the same arguments show that $\lim_{n\rightarrow\infty}I_{n,\mT}=
\e^{\frac{t^2}{2}\|a\|^2_L}$  uniformly for all $\mT\in\mR(K_n)$.

\subsection{Proof of Theorem \ref{thm4}}\label{sec:thm4_proof}
Let  $a_\mE$ denote the projection of $a$ onto $\mF[\mE]$ (the set of all forest mappings \eqref{forest} supported on a given ensemble $\mE$). 
The no-bias condition \eqref{eq:nobias} is satisfied automatically since $a_\mE=a$.

Similarly as in Rockova and van der Pas (2017) (Corollary 5.1), one can show that the posterior concentrates on   $\mE$, whose  trees are not too big (i.e. $\sum_{t}K^t\lesssim n\varepsilon_n^2/\log n$ for $\varepsilon_n=n^{-\alpha/(2\alpha+p)}\log n$). 
The next step is to show that the prior is sufficiently diffuse in the sense that  it does not change much under a small perturbation. To this end, we introduce a local shift, for some $s>0$,
$$
f_{\mE,\mB_s}=\sum_{t=1}^T \left(f_{\mT^t,\b^t}-\frac{s}{T\sqrt{n}}\right)= f_{\mE,\mB}-\frac{s}{\sqrt{n}},
$$
where $\mB=[\b^1_s,\dots,\b^T_s]$ and where $\b^t_s=\b^t-s/(T\sqrt{n})$.
 Now we have to perform the change of measures from $f_{\mE,\mB}$ to $f_{\mE,\mB_s}$. We start with the independent Laplace prior \eqref{eq:priorb_laplace} for each $\b^t$ with a penalty $\lambda_t\asymp 1/\sqrt{K^t}$.
 Similarly as in Section \ref{sec:laplace_proof}, we have
 $$
 \prod_{t=1}^T \pi(\b^t)<\prod_{t=1}^T \pi(\b^t_s)\times \exp\left(\frac{s}{T\sqrt{n}}\sum_{t=1}^T\lambda_t K^t\right)<\prod_{t=1}^T \pi(\b^t_s)\times \exp\left(\frac{s\,c_1}{\sqrt{n}}\max_t\sqrt{K^t}\right)
 $$
 for some $c_1>0$. Since $\max_t{K^t}\lesssim n\varepsilon_n^2$, the exponential term converges to one. A similar argument holds also for the lower bound.
 For the Gaussian prior $\b^t\C K^t\sim\mathcal{N}_{K^t}(0,K^t\times\mathrm{I}_{K^t})$,   we have
$$
\prod_{t=1}^T\pi(\b^t)=\prod_{t=1}^T\pi(\b^t_s)\times \exp\left(\frac{s^2}{2nT^2}+\frac{s\|\mB\|_2}{\sqrt{nT}}\right),
$$
where $\mB=(\b^{1\prime},\dots,\b^{T\prime})\in\R^{\sum_tK^t}$ is the vector of all step heights in the ensemble.  Similarly as before, one can show that the conditional posterior distribution of $\mB$, given $\mE$, concentrates around the projection of $f_0$ onto  $\mF[\mE]$ at the rate $\sum_t K^t\log n$. Since $\sqrt{\sum_t K^t\log n/n}\lesssim \sqrt{\varepsilon_n^2\log n}\rightarrow 0$, we can use similar arguments as in Section \ref{sec:proof:thm1} to conclude the BvM property.

\end{document}